\documentclass[12pt, reqno]{article}
\usepackage{amsmath}
\usepackage{amssymb}
\usepackage{amsthm}
\usepackage[dvipsnames,table,xcdraw]{xcolor}
\usepackage[colorlinks=true,
linkcolor=webgreen,
filecolor=webbrown,
citecolor=webgreen]{hyperref}

\definecolor{webgreen}{rgb}{0,.5,0}
\definecolor{webbrown}{rgb}{.6,0,0}

\newcommand{\seqnum}[1]{\href{https://oeis.org/#1}{\underline{#1}}}
\def\modd#1 #2{#1\ \mbox{\rm (mod}\ #2\mbox{\rm )}}

\theoremstyle{definition}
\newtheorem{definition}{Definition}
\theoremstyle{plain}
\newtheorem{lemma}{Lemma}
\newtheorem{proposition}{Proposition}
\newtheorem{claim}{Claim}
\newtheorem{theorem}{Theorem}
\newtheorem*{theorem*}{Theorem}
\newtheorem{example}{Example}

\begin{document}

\begin{center}
\vskip 1cm
    {\LARGE\bf Counting the Nontrivial Equivalence Classes of $S_n$ Under $\{1234,3412\}$-Pattern-Replacement}
    \vskip 1cm
    \large
    Quinn Perian
    
    Stanford Online High School
    
    Academy Hall Floor 2 8853
    
    415 Broadway
    
    Redwood City, CA 94063
    
    USA
    
    \href{mailto:quinn.perian@outlook.com}{\tt quinn.perian@outlook.com}
    
    \ \\Bella Xu
    
    William P. Clements High School
    
    4200 Elkins Drive
    
    Sugar Land, TX 77479
    
    USA
    
    \href{mailto:bxu107@gmail.com}{\tt bxu107@gmail.com}
    
    \ \\Alexander Lu Zhang
    
    Lower Merion High School
    
    315 E Montgomery Ave
    
    Ardmore, PA 19003
    
    USA
    
    \href{mailto:azhang896@gmail.com}{\tt azhang896@gmail.com}
\end{center}

\begin{abstract}
We study the $\{1234, 3412\}$ pattern-replacement equivalence relation on the set $S_n$ of permutations of length $n$, which is conceptually similar to the Knuth relation.  In particular, we enumerate and characterize the nontrivial equivalence classes, or equivalence classes with size greater than 1, in $S_n$ for $n \geq 7$ under the $\{1234, 3412\}$-equivalence.  This proves a conjecture by Ma, who found three equivalence relations of interest in studying the number of nontrivial equivalence classes of $S_n$ under pattern-replacement equivalence relations with patterns of length $4$, enumerated the nontrivial classes under two of these relations, and left the aforementioned conjecture regarding enumeration under the third as an open problem.

\end{abstract}

\section{Introduction}

A permutation $\pi \in S_n$ is said to \emph{contain a pattern}
$\sigma \in S_c$, $c \le n$, if there is a $c$-letter subsequence
$\pi_{i_1}, \pi_{i_2}, \ldots, \pi_{i_c}$ of $\pi$ that is
\emph{order-isomorphic} with $\sigma$ (i.e., for indices
$j, k \in [c]$, $\pi_{i_j} < \pi_{i_k}$ if and only if
$\sigma_j < \sigma_k$).

In the past thirty years, the topic of permutation patterns has risen
to the forefront of combinatorics (see Kitaev \cite{kitaev2011patterns} for a survey) and has even spawned its own annual conference \cite{conf}.

The focus of this paper is \emph{permutation pattern-replacement
 equivalences}. Given a set of patterns $P \subseteq S_c$ and a
permutation $\pi \in S_n$, one can perform a
\emph{$P$-pattern-replacement} on $\pi$ by taking a subsequence
$\pi_{i_1}, \ldots, \pi_{i_c}$ of $\pi$ that forms a pattern
$\sigma \in P$ and rearranging the relative order of the characters
$\pi_{i_1}, \ldots, \pi_{i_c}$ so that they form a different pattern
in $P$. We say that two permutations $\alpha, \beta \in S_n$ are
\emph{$P$-replacement equivalent} if $\alpha$ can be reached from
$\beta$ by a series of $P$-pattern-replacements. This defines an
equivalence relation on $S_n$, which is known as the $P$-replacement
equivalence \cite{kuszmaul2013counting, kuszmaulzhou20, S12, LPRW,
 PRW, Kn, NS, ma, fazel2014equivalence}.

The study of permutation pattern-replacement equivalences is closely
related to the study of permutation pattern avoidance
\cite{kitaev2011patterns,SS,claesson2001generalized, lewis2011pattern,
 marcus2004excluded, albert2006stanley, bona2007new, bona2004limit},
which seeks to count the number of singleton equivalence classes
(i.e., the number of permutations containing no patterns in $P$). The
dual problem of counting the number of non-singleton (i.e.,
\emph{nontrivial}) equivalence classes has recently received a large
amount of attention in the literature \cite{kuszmaul2013counting,
 kuszmaulzhou20, S12, LPRW, PRW, Kn, NS, ma, fazel2014equivalence}.

Most of the research so far on permutation pattern-replacement
equivalences has worked to systematically understand the equivalence
classes for all pattern-replacement equivalence relations involving
patterns of length three \cite{kuszmaul2013counting, kuszmaulzhou20,
 S12, LPRW, PRW, Kn, NS}. Many interesting number sequences have
arisen (e.g., the Catalan numbers, Motzkin numbers, tribonacci
numbers, central binomial coefficients, and many more complicated
number sequences). 

Recently, Ma \cite{ma} initiated the systematic study of
pattern-replacement equivalences with patterns of length four. Because
the number of pattern-replacement sets $P \subseteq S_4$ is very
large, Ma took a computational approach to identifying which of the
pattern-replacement sets were most interesting to study. In
particular, Ma computed the number of nontrivial equivalence classes
under $P$-equivalence for all sets $P \in S_4 \times S_4$,
and then matched the resulting number sequences with the Online
Encyclopedia of Integer Sequences (OEIS) \cite{sloane2003online} in
order to identify pattern-replacement sets $P$ for which the number of
nontrivial equivalence classes has a natural formula. Ma identified
three formulas of particular interest and was able to
enumerate the nontrivial equivalence classes for the equivalence relations corresponding to two of
them. Enumerating the equivalence classes for the third equivalence
relation, namely the $\{1234, 3412\}$-equivalence, has
remained an open question.

In this paper, we resolve the aforementioned open problem by proving the following theorem:
\begin{theorem}[Conjectured by Ma \cite{ma}]
For $n \geq 7$, the number of nontrivial equivalence classes of $S_n$ under the $\{1234,3412\}$-equivalence is $\frac{n^3+6n^2-55n+54}{6}$ (given by sequence \seqnum{A330395}).
\label{thm:nontrivial}
\end{theorem}

Recall that we define two permutations $\alpha, \beta \in S_n$ to be equivalent
under the $\{1234, 3412\}$-equivalence if $\alpha$ can be reached from
$\beta$ by performing a series of $1234 \rightarrow 3412$ and
$3412 \rightarrow 1234$ pattern-replacements. A
$1234 \rightarrow 3412$ pattern-replacement in a permutation $\pi$
simply takes an increasing 4-letter subsequence
$\pi_{i_1}, \pi_{i_2}, \pi_{i_3}, \pi_{i_3}$ (where
$i_1 < i_2 < i_3 < i_4$), and places each of
$\pi_{i_3}, \pi_{i_4}, \pi_{i_1}, \pi_{i_2}$ in positions
$i_1, i_2, i_3, i_4$ of the permutation. For example, in the
permutation $\pi = 7162435$, the subsequence $1, 2, 3, 5$ forms a
$1234$-pattern, and we can perform a $1234 \rightarrow 3412$
pattern-replacement to obtain the new permutation
$\pi' = 7365412$. Similarly, a $3412 \rightarrow 1234$
pattern-replacement takes any four-letter subsequence that forms a
$3412$ pattern, and rearranges the letters in that subsequence to
instead be in increasing order, thereby forming a $1234$ pattern.

Theorem~\ref{thm:nontrivial} counts the number of \emph{nontrivial}
equivalence classes in $S_n$ under the $\{1234,
3412\}$-equivalence. These are the equivalence classes of size greater
than one, or alternatively, the equivalence classes consisting of
non-$\{1234, 3412\}$-avoiding permutations. In the remainder of the
paper, we prove Theorem~\ref{thm:nontrivial}.

\subsection{Formal definitions}

\begin{definition}[Standardization of a Permutation]
Given a word $\rho$ of length $n$ consisting of $n$ distinct letters in $\mathbb{N}$, the \textit{standardization} of $\rho$ is the permutation in $S_n$ obtained by replacing the $i$-th smallest letter in $\rho$ with $i$ for each $i \in \{1, 2, \ldots, n\}$.
\end{definition}

\begin{definition}[Sub-standardization of a Permutation]
Given a permutation $\pi \in S_n$, and a subword $\rho$ of $\pi$, the standardization of $\rho$ is known as a \textit{sub-standardization} of $\pi$. If removing the letter $a$ from $\pi$ gives $\rho$, then the standardization of $\rho$ is also called the sub-standardization of $\pi$ formed by removing $a$.
\end{definition}

\begin{definition}[Pattern Formation]
Given a pattern $p \in S_c$ and a permutation $\pi \in S_n$, we say that $\pi$ contains pattern $p$ if $p$ is a sub-standardization of $\pi$. If a subword $\rho$ of $\pi$ has standardization $p$, then we say that $\rho$ \textit{forms} pattern $p$.
\end{definition}

\begin{definition}[Pattern-Replacement]
 Given two patterns $p, q \in S_c$ and a permutation $\pi \in S_n$, a
 $p \rightarrow q$ \textit{pattern-replacement} can be performed by taking any
 subword $\rho$ of $\pi$ that forms pattern $p$, and rearranging the
 letters in $\rho$ to instead form pattern $q$.
\end{definition}

\begin{definition}[Pattern-Replacement Equivalences]
 Given a set of patterns $P \subseteq S_c$, we say that two
 permutations $p, q \in S_n$ are \textit{equivalent} under $P$-equivalence if
 $p$ can be reached from $q$ by a sequence of pattern-replacements
 using patterns in $P$. Maximal collections of $P$-equivalent
 permutations are known as \textit{equivalence classes}, and an equivalence
 class is said to be \textit{nontrivial} if it contains more than one
 permutation.
\end{definition}

\subsection{Paper outline}
To prove Theorem
\ref{thm:nontrivial}, we will use three lemmas in order to
characterize the equivalence classes under the
$\{1234, 3412\}$-equivalence. To begin, we use the principle of inclusion-exclusion to establish a recurrence
relation that allows us to express the equivalence classes in
$S_n$ in terms of the equivalence classes in $S_{n - 1}$ and $S_{n-2}$ (Section
\ref{sec:recursion}). This reduces the proof of Theorem
\ref{thm:nontrivial} to proving that the number of equivalence
classes of a certain form is $n + 1$. We then characterize $n - 1$ of
these classes that have a natural combinatorial structure, consisting of
all permutations one transposition away from one of $n-1$ different
representative permutations (Section \ref{sec:small}). Finally, we
show that the remaining permutations fall into exactly two
classes depending on the parity of the number of inversions they
contain (Section \ref{sec:parity}); here, we use a proof structure
that exploits the \emph{stooge-sort technique} of Kuszmaul
\cite{kuszmaul2013counting}. Combining these steps, we prove Theorem
\ref{thm:nontrivial}.

\section{Reducing to permutations neither beginning with \texorpdfstring{$n$}{} nor ending with 1}\label{sec:recursion}

In this section, we reduce the proof of Theorem~\ref{thm:nontrivial} to counting the number of nontrivial equivalence classes that contain no permutations with $n$ in the first position or $1$ in the last position.
\begin{lemma}\label{lem:reduce with recursion}
Theorem~\ref{thm:nontrivial} can be reduced to proving that there are exactly $n+1$ nontrivial equivalence classes of $S_n$ where no permutation begins with $n$ or ends with $1$.
\end{lemma}
\begin{proof}
Let $A_n$ be the total number of nontrivial equivalence classes of $S_n$ under the $\{1234,3412\}$-equivalence. Let $B_n$ be the number of nontrivial equivalence classes of $S_n$ under the $\{1234,3412\}$-equivalence that do not contain any permutation beginning with $n$ or ending with $1$. Then we have the following recurrence relation.
\begin{claim}
For $n \ge 3$, 
\begin{equation}\label{eq:1}
A_n - B_n = 2A_{n-1}-A_{n-2}.
\end{equation}
\label{clm:recurrence}
\end{claim}
\begin{proof} 
Since the largest element is never first and the smallest element is never last in the patterns $1234$ and $3412$, no pattern-replacements under the $\{1234,3412\}$-equivalence can move a leading $n$ or an ending $1$. Therefore, appending $n$ to the front (resp., $1$ to the end) of all permutations of an equivalence class in $S_{n - 1}$ results in an equivalence class in $S_n$. Call this the \emph{lifting observation}.

We can count the number of equivalence classes where all permutations have either $n$ in the first position or $1$ in the last position (or both) in two different ways. The left side of \eqref{eq:1} represents this number with complementary counting, where all equivalence classes not satisfying the requisite are subtracted from the total number of equivalences classes. In the right side of \eqref{eq:1}, we count the same number of equivalence classes using the principle of inclusion-exclusion. By the lifting observation, we can append an $n$ to the front of all permutations in $S_{n-1}$ or append a $1$ to the back of all permutations in $S_{n-1}$ to form all possible equivalence classes where all permutations have either $n$ in the first position or $1$ in the last position (or both). (This contributes $2A_{n - 1}$ to \eqref{eq:1}.) The repeats, which are formed by appending both an $n$ to the front and a $1$ to the back of all permutations in $S_{n-2}$, are then subtracted off. (This removes $A_{n - 2}$.) 

Since the two sides of \eqref{eq:1} count the same quantity, we must have equality.
\end{proof}

Using the recursion given in the preceding claim, we now analyze the value that $B_n$ must take in order for $A_n$ to have the formula stated in Theorem~\ref{thm:nontrivial}.

\begin{claim}
Showing that $B_n = n+1$ suffices to prove that $A_n = \frac{n^3+6n^2-55n+54}{6}$ for all $n \ge 7$.
\end{claim}
\begin{proof}
Suppose $B_n = n + 1$ for all $n \ge 7$. We use this to prove that $A_n = \frac{n^3+6n^2-55n+54}{6}$ for all $n \ge 7$.

The base cases of $n=7,8$ for $A_n = \frac{n^3+6n^2-55n+54}{6}$ have been checked by computer \cite{ma}. 

Assume by induction that our formula for $A_n$ holds for $n = k-1$ and $n = k-2$ for some integer $k \ge 9$. Then by Claim~\ref{clm:recurrence}, 
\begin{align*}
A_k &= B_k + 2A_{k-1} - A_{k-2} \\
&= k+1 + 2\cdot\frac{(k-1)^3+6(k-1)^2-55(k-1)+54}{6} \\
& \phantom{foobargoobar} - \frac{(k-2)^3+6(k-2)^2-55(k-2)+54}{6} \\
&= k+1 + 2\cdot\frac{k^3+3k^2-64k+114}{6} - \frac{k^3-67k+180}{6} \\
&= \frac{k^3+6k^2-55k+54}{6},
\end{align*}
which is our formula for $A_k$. 

By induction, $A_n = \frac{n^3+6n^2-55n+54}{6}$ holds for all $n \ge 7$.
\end{proof}

This completes the proof of the lemma.
\end{proof}

\section{Characterizing the \texorpdfstring{$n-1$}{} small classes}\label{sec:small}

In this section we characterize $n-1$ of the nontrivial equivalence classes of $S_n$ that contain no permutation beginning with $n$ or ending with $1$. Each of these classes is associated with a ``leader permutation'' that is one transposition away from each permutation in the class. 

Below we define the notion of a leader permutation and what we mean when we call two permutations ``adjacent.'' The equivalence classes then correspond to the sets of permutations that are adjacent to each leader permutation.

\begin{definition}[Leader Permutation] Define a \textit{leader permutation} of length $n$ to be a permutation of the form $a_1a_2 \cdots a_n$ such that for some integer $k \in [2,n]$, $a_i = k-i$ for all $1\leq i < k$ and $a_{i} = n+k-i$ for all $k \leq i \leq n$.
\end{definition}
\begin{example}
The $7$ leader permutations of length $8$ are $18765432$, $21876543$, $32187654$, $43218765$, $54321876$, $65432187$, and $76543218$. 
\end{example}
\begin{definition}[Adjacent Permutations] 
We define two permutations $\pi,\rho\in S_n$ to be \textit{adjacent} if neither $\pi$ nor $\rho$ begins with $n$ or end with $1$ and if $\pi$ is a transposition of $\rho$ such that the positive difference between the two letters in the transposition is neither $1$ nor $n-1$. 
\end{definition}

We now provide our characterization of the equivalence classes associated with the leader permutations in $S_n$.

\begin{proposition}
For any leader permutation $\pi$, the set of permutations adjacent to $\pi$ is an equivalence class. Moreover, there are $n-1$ of these classes.
\label{prop:adjacentclasses}
\end{proposition}

To prove Proposition~\ref{prop:adjacentclasses}, we begin by proving that the equivalence classes are disjoint.
\begin{lemma}\label{lem: disjoint adj to leader}
For distinct leader permutations $\pi, \psi$, the set of permutations adjacent to $\pi$ is disjoint from the set of permutations adjacent to $\psi$.
\label{lem:adjacent1}
\end{lemma}
\begin{proof}
Assume for the sake of contradiction that the two sets share a common element. This would imply that it is possible to reach $\pi$ from $\psi$ in two transpositions. Since any leader permutation is a derangement of any other leader permutation, the permutations $\pi$ and $\psi$ differ in all $n \geq 7$ positions. However, since each transposition can only switch the positions of two letters, it is impossible for two transpositions to get $\pi$ from $\psi$, a contradiction.
\end{proof}

Next we prove that 
\begin{lemma}
If we have a permutation adjacent to some leader permutation $\pi$ and we apply a pattern-replacement to it, the resulting permutation is also adjacent to $\pi$. 
\label{lem:adjacent2}
\end{lemma}
\begin{proof}
Note that $\pi$ consists of a decreasing set of consecutive letters followed by another decreasing set of consecutive letters that are all larger than those of the first set; let $A$ and $B$ denote these two sets, respectively. 

Consider a permutation $a_1a_2 \cdots a_n$ such that if a transposition is applied to two letters $a_i$ and $a_j$, $i \neq j$, the resulting permutation is equal to a leader permutation $\pi$. We claim that any pattern-replacement applied to $a_1a_2 \cdots a_n$ must use both $a_i$ and $a_j$ and result in a permutation adjacent to $\pi$. We show this using casework based on which type of pattern in $a_1a_2 \cdots a_n$ is replaced. 

\textbf{Case 1:} A 1234 pattern is replaced. Since $A$ and $B$ are decreasing, the longest increasing subsequence in $\pi$ has length 2, so any 1234 pattern in $a_1a_2 \cdots a_n$ must use both letters $a_i$ and $a_j$ involved in the transposition. Moreover, the two swapped letters $a_i$ and $a_j$ would need to be in the same set (i.e., both in $A$ or both in $B$), because swapping a letter in $A$ with a letter in $B$ would not place the two letters in increasing order. In order to have an increasing sequence of length 4, the pattern must additionally contain one letter between $a_i$ and $a_j$, and one letter from the set $A$ or $B$ that does not contain $a_i$ and $a_j$. In particular, if the two letters besides $a_i$ and $a_j$ are both from $A$ or are both from $B$ then they are in decreasing order (a contradiction). If either of the two letters is from the same set (i.e., out of $A$ and $B$) as $a_i, a_j$ are in but is not between $a_i$ and $a_j$ then it does not form an increasing sequence with $a_i$ and $a_j$ (a contradiction). Since a $\{1234,3412\}$-pattern replacement swaps the first and third letters and swaps the second and fourth letters, $a_i$ and $a_j$ are swapped back to their original order in $\pi$, and a new transposition is applied. The new permutation is adjacent to $\pi$.

\textbf{Case 2:} A 3412 pattern is replaced.  
Consider the four three-term subpatterns of the 3412 pattern, 341, 342, 312, and 412; we show that $\pi$ cannot contain any of these. Assume for the sake of contradiction that some three letters $b_x$, $b_y$, $b_z$, $x<y<z$, of $\pi$ follow one of these patterns. Then, in any of the patterns, $b_x>b_z$. In addition, note that $b_y$ is either the largest or the smallest letter. If $b_y$ is the largest letter, then $b_x<b_y$, so $b_x$ must be in $A$ and $b_y$ must be in $B$, implying that $b_z$ must be in $B$ as well. Similarly, if $b_y$ is the smallest letter, then $b_y<b_z$, so $b_y$ must be in $A$ and $b_z$ must be in $B$, and thus $b_x$ must be in $A$. In either case, $b_x$ is in $A$ and $b_z$ is in $B$, which is a contradiction because $b_x>b_z$. Thus $\pi$ cannot contain any length-3 subpattern of $3412$.

As a consequence, any 3412 pattern in $a_1a_2 \cdots a_n$ would have to contain both $a_i$ and $a_j$. Furthermore, $a_i$ and $a_j$ would need to be in different sets, and the 3412 pattern would need to contain exactly one letter between $a_i$ and $a_j$. To see this, note that in $\pi$, any four term subsequence corresponds to one of the patterns 4321, 3214, 2143, or 1432 depending on how many letters are from each of $A$ and $B$. The only transpositions of two terms that could result in 3412 are swapping the 2 and 4 in 3214 and swapping the 1 and 3 in 1432. In either case, the smaller term is from $A$ and the larger term is from $B$, and there is exactly one term between these two terms.

Since a $\{1234,3412\}$-pattern-replacement swaps the first and third letters and swaps the second and fourth letters, $a_i$ and $a_j$ are swapped back to their original order in $\pi$, and a new transposition is applied. The new permutation is adjacent to $\pi$.

This completes the proof of the claim.
\end{proof}

Finally, we show that any two permutations adjacent to the same leader permutation are always equivalent.
\begin{lemma}
For any permutations $\tau$ and $\sigma$ that are both adjacent to a leader permutation $\pi$, $\tau$ can be reached from $\sigma$ by a sequence of $\{1234,3412\}$-pattern-replacements. 
\label{lem:adjacent3}
\end{lemma}
\begin{proof}
We prove this using induction on $n$. Our base cases of $n=7$ and 8 can be checked by computer. Assume as an inductive hypothesis that the result holds for $n - 1 \ge 8$. Consider two permutations $\tau$ and $\sigma$ of length $n$ that are adjacent to some leader permutation $\pi$. Pick one of their common letters $a$ that is (1) not $1$ or $n$; (2) not part of either transposition from the leader permutation; (3) not adjacent in position to both letters in the transposition from $\pi$ to $\tau$; and (4) not adjacent in position to both letters in the transposition from $\pi$ to $\sigma$. Such an $a$ is guaranteed to exist because $n \geq 9$. Now, consider the sub-standardizations $\tau \prime$, $\sigma \prime$, and $\pi \prime$ respectively formed by removing $a$ from $\tau$, $\sigma$, and $\pi$. Note that $\pi \prime$ is a leader permutation of length $n-1$ (because $a \neq 1$ and $a \neq n$). In addition, $\tau \prime$ and $\sigma \prime$ are both adjacent to $\pi \prime$. Thus, by our inductive hypothesis, $\pi \prime$ can be reached from $\tau \prime$ by a sequence of pattern-replacements. By looking at the letters in $\tau$ corresponding to those of $\tau\prime$, we see that $\pi$ can be reached from $\tau$ by a sequence of pattern-replacements. This completes the proof by induction. 
\end{proof}

Combined, Lemmas \ref{lem:adjacent1}, \ref{lem:adjacent2}, and \ref{lem:adjacent3} complete the proof of Proposition~\ref{prop:adjacentclasses}.

\section{Showing that there are two remaining classes}\label{sec:parity} 

We now proceed to show that there are only two nontrivial classes remaining to be analyzed. To do this, we will find two representative permutations, one from each class. First, we will show that these representative permutations cannot be in the same equivalence class because they differ in parity and, hence, that there are at least two remaining equivalence classes. Then, we will show that every permutation that we have yet to analyze is equivalent to one of these two representative permutations, hence showing that there are at most, and thus exactly, two remaining nontrivial equivalence classes. To show the second part, we will use induction on the size of the permutation. We will begin by removing one letter of the permutation and applying the inductive hypothesis to the remaining letters to transform them into a representative permutation (of length $n - 1$). We then select a different letter in the new permutation (specifically, we pick a letter that is already in the ``final'' position that we want it to be in), we remove that letter from the permutation, and we apply the inductive hypothesis on the remaining letters in order to transform them into a representative permutation (of length $n - 1$). If done correctly, this construction causes the entire final permutation to form a representative permutation of length $n$, completing the proof.

We start out with some preliminary definitions.

\begin{definition}[Primary Permutation]
We call a permutation in $S_n$ \textit{primary} if it is not adjacent to a leader permutation and does not start with $n$ or end with 1.
\end{definition}
\begin{definition}[Primary Class] We define a \textit{primary class} to be a nontrivial equivalence class of primary permutations in $S_n$.
\end{definition} 

\begin{lemma}\label{lem: parity invariance}
All of the permutations in each equivalence class have the same parity.
\end{lemma}
\begin{proof}
Each $1234 \leftrightarrow 3412$ pattern-replacement consists of two transpositions. Since parity is invariant under both transpositions, the lemma follows.

\end{proof}

In the remainder of the section, we will attempt to start by applying induction. To do this, we will look at a sub-standardization of an arbitrary primary permutation formed by excluding one letter. The goal is to find conditions on the letter removed such that the sub-standardization is also primary, hence allowing us to apply the inductive hypothesis to the sub-standardization. There are two possible problems that we would run into: either the sub-standardization is adjacent to a leader permutation, or the sub-standardization could start with its largest letter or end in $1$. The latter is easily dealt with and we will put off to Case $2$ of Lemma~\ref{lem: 2 primary classes}. The former, however, can be dealt with by an intuitive claim that is proven in Lemma~\ref{lem:creating-primary-permutations}. We show in Lemma~\ref{lem:creating-primary-permutations} that, if we remove some letter and end up with a permutation adjacent to a leader permutation, we can instead choose a different letter with some constraints to yield a permutation that is not adjacent to a leader permutation.

\begin{lemma}[Creating Primary Permutations]\label{lem:creating-primary-permutations}
Let $\rho\in S_k$ with $k\geq 5$ be a primary permutation. Suppose that there is a letter $a$ in $\rho$ such that removing $a$ results in a sub-standardization $\rho_{-(a)}$ that is leader-permutation-adjacent. Let $\tau_1$ be the transposition taking $\rho_{-(a)}$ to a leader permutation, and suppose $b$ is some letter in both $\rho$ and $\rho_{-(a)}$ that is not operated on by $\tau_1$ and such that $a \not\equiv b \pm \modd{1} {k}$. 
Also, suppose that if $\tau_1$ operates on two letters with only one letter of $\rho_{-(a)}$ between them, then $b$ is not that letter. Then the sub-standardization $\rho_{-(b)}$ obtained by removing $b$ from $\rho$ is not adjacent to any leader permutation.

\end{lemma}
\begin{proof}
 Suppose, for the sake of contradiction, that $\tau_2\cdot \rho_{- (b)}$ is a leader permutation for some transposition $\tau_2$. Let $\rho_{-(a,b)}$ be the sub-standardization of $\rho$ that contains all letters of $\rho$ other than $a$ and $b$ (here we say that a transposition of a sub-standardization swaps the same physical letters as it did in the original permutation). 
 
 Define a \emph{semi-leader permutation} to be a permutation that is either a leader permutation or the decreasing permutation. Note that for any semi-leader permutation, if we remove one letter and take the standardization, the result is always a semi-leader permutation.
 
 We claim that $\tau_1 \cdot \rho_{-(a, b)}$ is a semi-leader permutation (note that $\tau_1$ does not involve $a$ or $b$). In particular, $\tau_1 \cdot \rho_{-(a)}$ is a leader permutation, and thus removing $b$ to get $\tau_1 \cdot \rho_{-(a, b)}$ gives a semi-leader permutation. Moreover, $\tau_1 \cdot \rho_{-(a, b)}$ is the unique semi-leader permutation adjacent to $\rho_{-(a, b)}$. In particular, any other such semi-leader permutation $x$ would differ from $\tau_1 \cdot \rho_{-(a, b)}$ in at most $4$ positions, and would therefore have to agree with $\tau_1 \cdot \rho_{-(a, b)}$ in at least one position. Whenever two semi-leader permutations of the same length agree in at least one position, they must be the same permutation, meaning that $x$ actually equals $\tau_1 \cdot \rho_{-(a, b)}$.

 If $\tau_2$ operates only on elements of $\rho_{-(a, b)}$ (i.e., $\tau_2$ does not involve $a$), we can conclude that $\tau_2=\tau_1$. In particular, $\tau_2 \cdot \rho_{-(a, b)}$ is a semi-leader permutation because $\tau_2 \cdot \rho_{-(b)}$ is a leader permutation. Thus $\tau_1 \cdot \rho_{-(a, b)} = \tau_2 \cdot \rho_{-(a, b)}$, so $\tau_1=\tau_2$.

 However, if $\tau_1=\tau_2$ then $\tau_1$ takes both $\rho_{-(b)}$ and $\rho_{-(a)}$ to leader permutations. We claim that this then implies that $\rho$ is also a leader permutation, a contradiction. This is because when we add in $b$ to $\rho_{-(b)}$ to get $\rho$, $b$ is not adjacent to $a$ (recall that $a \not\equiv b \pm \modd{1} {k}$), so the characters that $b$ is adjacent to in $\rho_{-(a)}$ remain adjacent to $b$ in $\rho$ (although their values may be shifted by $1$), ensuring that $\rho$ is a leader permutation. 
 
 It remains to consider the case where $\tau_2$ operates on $a$. In this case, we define $\tau_2 \cdot \rho_{-(a, b)}$ to be the standardization of $\tau_2 \cdot \rho_{-(b)}$ with $a$ removed. Now $\tau_2$ moves only one letter in $\rho_{-(a, b)}$ and results in a semi-leader permutation (as $\tau_2 \cdot \rho_{-(a, b)}$ is contained in $\tau_2 \cdot \rho_{-(b)}$, which is a semi-leader permutation). Thus there are two semi-leader permutations $x = \tau_1 \cdot \rho_{-(a, b)}$ and $y = \tau_2 \cdot \rho_{-(a, b)}$, each of length $k - 2$, such that we can get from $x$ to $y$ by swapping two letters (call this transposition $\tau\prime$) and then moving one letter. We place each letter of the permutation on a circle in clockwise order starting with the first letter. Observe that, up to rotation, all semi-leader permutations are equivalent when put on a circle in this way. Hence, on the circle $x$ and $y$ are equivalent up to rotation. Let $\tau\prime$ swap letters $j,k$. Note that, as $\tau\prime$ does not swap adjacent letters, neither $j$ nor $k$ are adjacent to any of the same letters as they were before $\tau\prime$. However, moving one letter can remove at most three adjacencies (2 containing the moved letter, 1 for the letters it moved between), so one of the four new adjacencies with $j$ or $k$ remains. Hence, after performing transposition $\tau\prime$ and moving one letter, the circle is not the same up to rotation, a contradiction.

\end{proof} 
\begin{lemma}\label{lem: 2 primary classes}
There are exactly two \textbf{primary classes} in $S_n$ for $n\geq 7$.
\end{lemma}

\begin{proof}
Let $n \ge 7$. Define $\Pi_n=123\cdots n$ and $\Psi_n =123\cdots (n-3)(n-2)n(n-1)$. First of all, note that the permutations $\Pi_n$ and $\Psi_n$ are both primary permutations. In particular, for $n\geq 7$, no transposition takes $\Pi$ or $\Psi$ to a leader permutation because for any transposition $\tau$, both $\tau \cdot \Pi$ and $\tau \cdot \Psi$ contain $123$ patterns. Further, the parity of the number of inversions is even in $\Pi$ and odd for $\Psi$, so they are in two distinct primary classes by Lemma~\ref{lem: parity invariance}. This proves that $2$ is a lower bound on the number of primary classes.

To proceed, we use induction to show $2$ is an upper bound on the number of primary classes.

The base cases of $n=7,8,9,10$ are verifiable by computer.

For the inductive step, we assume that Lemma~\ref{lem: 2 primary classes} holds for an $n=k-1$ with $k-1\geq 9$ and prove that Lemma~\ref{lem: 2 primary classes} holds for $n=k$. To do so, we show that any primary permutation in $S_k$ is equivalent to one of $\Pi_k$ and $\Psi_k$.  Let $\rho\in S_k$ be an arbitrary primary permutation; we will show that $\rho$ is equivalent to either $\Pi_k$ or $\Psi_k$ depending on the parity of the number of inversions $\rho$ contains.  Because primary classes are nontrivial, we can find some $1234$ or $3412$ pattern in $\rho$. Take an arbitrary such pattern $p$. As $k\geq 10$, we can now find some number $a$ inside of $\rho$ such that $a\neq 1$, $a \neq k$, $a$ is neither the first nor last number in $\rho$, and $p$ does not contain $a$. Now, we look at the sub-standardization $\rho\prime\in S_{k-1}$ of $\rho$ that excludes $a$. As $a$ was not in $p$, $\rho\prime$ still contains $p$ and thus is part of a nontrivial class. As $a$ was neither $n$ nor the first number in $\rho$, $\rho\prime$ does not start with $n$. Similarly, $\rho\prime$ does not end with $1$. We proceed with casework on whether $\rho\prime$ is adjacent to a leader permutation.

\textbf{Case 1:} $\mathbf{\rho\prime}$ \textbf{is a primary permutation.} In this case, by the inductive hypothesis, $\rho\prime$ is equivalent to either $\Pi_{k-1}$ or $\Psi_{k-1}$. Hence, in $\rho$, we can use the equivalence relation to re-order the elements excluding $a$ to form $\Pi_{k-1}$ or $\Psi_{k-1}$. Call this new permutation after the re-ordering $\rho_1$. Now, observe that because $a$ was not the first number in $\rho$ but both $\Pi_{k-1}$ and $\Psi_{k-1}$ begin with $1$, $1$ is the first number in $\rho_1$. Moreover, the second number in $\rho_1$ is either $2$ or $a$, neither of which equals $n$.

Consider the sub-standardization $\rho_1\prime$ created by excluding the $1$ in $\rho_1$. Note that $\rho_1\prime$ neither starts with $n$ nor ends with $1$. Additionally, we attain a $1234\cdots(k-4)$ pattern in $\rho_1\prime$ by ignoring the $a$ and the final two numbers of $\rho_1\prime$. Such a pattern clearly must contain a $1234$ pattern and can never occur in a permutation adjacent to a leader permutation as $k-4\geq5$, so $\rho_1\prime$ is primary. Since $\rho_1\prime$ is primary, by our inductive hypothesis it is equivalent to either $\Psi_{k-1}$ or $\Pi_{k-1}$. In either case rearranging the sub-permutation of $\rho_1$ that $\rho_1\prime$ corresponds to yields that $\rho_1$, and hence $\rho$, must be equivalent to either $\Psi_k$ or $\Pi_k$, as desired.

\textbf{Case 2:} $\mathbf{\rho\prime}$ \textbf{is not primary.} In this case, consider some letter $b$ in $\rho$ such that 
\begin{enumerate}
 \item if the second letter of $\rho$ is $k$, $b$ is not the first letter in $\rho$
 \item if the second to last letter of $\rho$ is $1$, $b$ is not the last letter in $\rho$ 
 \item if the last letter in $\rho$ is $2$, $b\neq 1$ 
 \item if the first letter in $\rho$ is $k-1$, $b\neq k$.
\end{enumerate}

Next we show that, as long as $k \ge 11$, it is possible to select $b$ satisfying conditions (1), (2), (3), and (4) as well as the conditions of \ref{lem:creating-primary-permutations} (which we will reiterate shortly). Let $\tau$ be the transposition taking $\rho'$ to a leader permutation.  Note that these conditions imply that the sub-standardization $\rho^\star\in S_{k-1}$ attained by removing $b$ from $\rho$ neither ends in $1$ nor starts with $k$. Further, note that if the if-conditions for $(2)$ and $(3)$ both take effect, this would force $\rho$ to end with $12$, and either the $1$ or the $2$ would be operated on by $\tau$ as no leader permutation ends with $12$. Similarly, if the if-conditions for $(1)$ and $(4)$ both take effect, $\rho$ would start with $(k-1)k$, and either $(k-1)$ or $k$ would be operated on by $\tau$. Also, it is easy to see that both letters operated by $\tau$ must be a part of any $1234$ or $3412$ pattern in $\rho'$, because the 3-letter subsequences of these patterns do not appear in leader permutations. Hence, we if we take $b$ satisfying (1), (2), (3), and (4) and we require $b$ to satisfy the conditions of Lemma~\ref{lem:creating-primary-permutations} (meaning that $b$ is not contained in the pattern $p$; $b$ is not operated on by $\tau$; $b \neq a$; $b\not\equiv a\pm \modd{1} {k}$; and if $\tau$ swaps two letters separated by only one letter, $b$ is not that letter) then in total we exclude at most $10$ elements from the possible choices (in particular, (1), (2), (3), and (4) exclude at most two elements not already excluded as part of $\tau$). This means that as long as $k\geq 11$, we can find such a $b$ in $\rho$.

Then, by Lemma~\ref{lem:creating-primary-permutations}, the permutation pattern formed by removing $b$ from $\rho$ must be a primary permutation, so this case reduces to Case $1$.

These two cases complete the induction, establishing that the number of primary classes is exactly 2.

\end{proof}
\section{Putting the proof together}
We now finish proving the theorem by putting the lemmas proved in the previous sections together. 
\begin{theorem*}[Theorem~\ref{thm:nontrivial} restated]
For $n \ge 7$, the number of nontrivial equivalence classes of $S_n$ under the $\{1234,3412\}$-equivalence is $\frac{n^3+6n^2-55n+54}{6}$.
\end{theorem*}
\begin{proof}
First of all, by Lemma~\ref{lem:reduce with recursion}, we see that it is sufficient to show that there are exactly $n+1$ classes of permutations in $S_n$ not beginning with $n$ or ending with $1$. Every such permutation is in exactly one of two categories, the set of permutations adjacent to leader permutations and primary permutations. By Proposition~\ref{prop:adjacentclasses} there are $n-1$ classes of permutations adjacent to a leader permutation, and by Lemma~\ref{lem: 2 primary classes} there are exactly two primary classes. Thus, in total we find that permutations not beginning in $n$ or ending in $1$ fit into $n-1+2=n+1$ classes, as desired. By Lemma~\ref{lem:reduce with recursion}, this completes the proof.
\end{proof}

\section{Acknowledgments}
The authors would like to thank Canada/USA Mathcamp for providing the opportunity to conduct this research. The authors would also like to thank William Kuszmaul for his mentorship on the project and for suggesting the research problem.

\bibliographystyle{jis.bst}
\bibliography{Bibliography}

\bigskip
\hrule
\bigskip

\noindent 2020 {\it Mathematics Subject Classification}:
Primary 05A05; Secondary 05A15.

\noindent \emph{Keywords:} permutation, permutation pattern, equivalence class, integer sequence, pattern-replacement, transposition.
\bigskip
\hrule
\bigskip
\noindent (Concerned with sequence \seqnum{A330395}.) 
\end{document}